\documentclass[clean,reqno,twoside,a4paper,11pt]{amsart}
\usepackage[american]{babel}
\usepackage{mltools}

\usepackage{math62,mlthm10-REAR} 

\usepackage{upref,amsmath,amssymb,amsthm,mathrsfs,amsfonts}
\usepackage[colorlinks=true,linkcolor=blue,citecolor=blue]{hyperref}

\usepackage[symbol]{footmisc}
\usepackage{perpage}
\MakePerPage[2]{footnote} 

\usepackage[scaled]{helvet} 
\usepackage{courier} 
\usepackage[mathbf]{euler}
\usepackage{mllocal}

\numberwithin{equation}{section}

\newcommand\nonc{noncommutative}

\newcommand\gvec[2][n]{({#2}_1,\ldots,{#2}_{#1})} 
\newcommand\glist[2][n]{{#2}_1,\ldots,{#2}_{#1}} 
\newcommand{\avec}[1]{(a_1,\ldots,a_{#1})}
\newcommand{\alist}[1]{a_1,\ldots,a_{#1}}

\newcommand{\specA}{\spec_{\sA}}

\begin{document}

\date{\today}

\title[Divided Differences and Multivariate Holomorphic Calculus]{%
Divided Differences and Multivariate Holomorphic Calculus}

\author{Luiz Hartmann}
\address{Departamento de Matem\'atica,
	Universidade Federal de S\~ao Carlos (UFSCar),
	Brazil}
\email{hartmann@dm.ufscar.br, luizhartmann@ufscar.br}
\urladdr{http://www.dm.ufscar.br/profs/hartmann}
\thanks{Partially support by FAPESP (2022/16455-6), PROBAL CAPES/DAAD (8881.700909/2022-01) and we
	gratefully
	acknowledge the financial
	support of the Hausdorff Center for Mathematics, Bonn.}

\author{Matthias Lesch}
\address{Mathematisches Institut,
	Universit\"at Bonn,
	Endenicher Allee 60,
	53115 Bonn,
	Germany}\email{ml@matthiaslesch.de, lesch@math.uni-bonn.de}
\urladdr{www.matthiaslesch.de, www.math.uni-bonn.de/people/lesch}



\begin{abstract}
We review the multivariate holomorphic functional calculus for tuples in a
\emph{commutative} Banach algebra and establish a simple ``na\"ive'' extension
to commuting tuples in a general Banach algebra. The approach is na\"ive in the
sense that the na\"ively defined joint spectrum maybe too big.  The advantage of
the approach is that the functional calculus then is given by a simple concrete
formula from which all its continuity properties can easily be derived.

We apply this framework to multivariate functions arising as divided differences of a
univariate function. This provides a rich set of examples to which our na\"ive
calculus applies. Foremost, we offer a natural and straightforward proof of the
Connes-Moscovici Rearrangement Lemma in the context of the multivariate
holomorphic functional calculus.  Secondly, we show that the Daletski-Krein type
noncommutative Taylor expansion is a natural consequence of our calculus. Also
Magnus' Theorem which gives a nonlinear differential equation for the $\log$ of
the solutions to a linear matrix ODE follows naturally and easily from our
calculus.  Finally, we collect various combinatorial related formulas.
\end{abstract}

\maketitle

\tableofcontents


\section{Introduction}
\label{s.Intro}

This paper is a follow-up of \cite{Les2014}, which together with
\cite{ConMos2011, LesMos2015} serves as the primary motivation for the present
work. \cite{Les2014} shows that some of the rather involved combinatorial
formulas arising in the Spectral Geometry of noncommutative tori
\cite{ConMos2011, LesMos2015} find their natural explanation in terms of the
classical calculus of divided differences.  Furthermore, a more conceptual proof
of the Connes-Moscovici Rearrangement Lemma \cite[Sec. 6.2]{ConMos2011} in the
formalism of divided differences and a multivariate smooth functional calculus
was given.

An important pattern are expressions of the form
\[
   F(a\pup 0;,\ldots,a\pup n;)(b_1\cldots b_n),
\]
where $F$ is a multivariate function applied to algebra elements $a\pup
0;,\ldots,a\pup n;$ in (a topological tensor product) $\sAn[n+1]$ of a Banach
algebra \emph{paired} with the algebra elements $b_1,\ldots, b_n$.  The pairing
is most naturally defined if $\sAn[n+1]$ is equipped with the projective tensor
product topology (denoted by $\sAn[n+1]_\pi$).  The crucial property of the
latter which is used extensively is that in this case the multiplication map
\begin{equation} \label{eq-intro-1}
  \mu_{n+1}: \sAn[n+1]_\pi\longrightarrow \sA,\; a_0\otidots a_n
  \mapsto a_0\cldots a_n
\end{equation}
is continuous. In \cite{Qui1985}, Quigg characterizes the von Neumann algebras for which $\mu_{1}$ is continuous relative to the spatial $C^*$-norm (see \cite[Theorem 4.6]{Qui1985}). In particular, Quigg proved that $\mu_{1}$ is not continuous when $\sA$ is the algebra of bounded linear operators on a separable infinite-dimensional Hilbert space (see \cite[Lemma 4.4]{Qui1985}).

Lacking a continuous functional calculus in the Banach algebra $\sAn[n+1]_\pi$,
in \cite{Les2014} we circumvented the issues by establishing a functional
calculus, if $\sA$ is $C^*$, for smooth multivariate functions; it is a
$*$-homomorphism from the involutive Fr\'echet algebra $C^\infty(U^{n+1})$ to
the involutive Banach algebra $\sAn[n+1]_\pi$ (\cite[Theorem 3.2]{Les2014}).
This approach exploited the fact that for an open subset $U\subset \R^n$ the
algebra $C^\infty(U)$ with its natural Fr\'echet topology is nuclear and hence
$C^\infty(U^{n+1})$ is naturally isomorphic to the projective as well as the
injective tensor product of $n+1$ factors of $C^\infty(U)$.

While this approach worked, for several reasons it seems a bit unnatural.
Firstly, it is a little odd that relatively deep (though now classical) results
about nuclear spaces are needed.  Secondly, it requires $\sA$ to be a
$C^*$-algebra, and the functional calculus is a ``fake'' version of the
continuous functional calculus.  On the other hand the projective tensor product
$\sAn[n+1]_\pi$ is defined for any Banach algebra and therefore using just the
multivariate \emph{holomorphic} functional calculus seems most natural. While
holomorphicity is certainly narrower than smoothness, for the Rearrangement
Lemma holomorphic functions (even rational ones) suffice and we have the benefit
that the theory can be developed naturally in the context of Banach algebras.

This is the starting point of the current paper. We take it also as an excuse to
review the quite intricate history of the multivariate holomorphic functional
calculus which is by far not just an afterthought to the one variable case and
which is normally not taught in standard functional analysis courses. While a
uniform theory of a multivariate holomorphic functional calculus remains
elusive, various approaches have been proposed
\cite{All1971,Tay1970,Gam1969,Zam1979}. In general, it is only well-defined in a
commutative Banach algebra. We show in Section \ref{ss.NHFC}, however, that
there is a na\"ive version of the multivariate holomorphic functional calculus
which has the benefit of being unambiguously defined for commuting tuples in a
noncommutative Banach algebra and which is given by a simple concrete formula
from which all its continuity properties can easily be derived (Theorem
\ref{thm:NHFC}).

We also mention several more recent approaches to the multivariate holomorphic
functional calculus. McIntosh and Pryde \cite{McP1987}
studied commuting operators on a Banach space embedded in a Banach module over
a Clifford algebra. Brian Jefferies explored related ideas in greater depth in
his monograph {\it Spectral Properties of Noncommuting Operators} \cite{Jef2004},
focusing on n-tuples of noncommuting operators. Finally, the
monograph by Colombo, Sabadini, and Struppa \cite{CSS2011}
proposes a novel approach to constructing a general functional
calculus for not necessarily commuting n-tuples of operators, as well as a
functional calculus for quaternionic operators. We mention these works here
for completeness; a detailed comparison is beyond the scope of our relatively
elementary presentation.

Divided differences of a holomorphic function naturally give rise to multivariate
holomorphic functions. For such functions, our na\"ive multivariate functional
calculus works seamlessly for commuting tuples in any Banach algebra. As
applications we show a noncommutative Newton interpolation formula,
noncommutative Taylor formulas in the tradition of Krein-Daletskii
\cite{Dal1990},\cite{Dal1998}, cf.  also the ``asymptotic analysis'' versions of
Paycha \cite{Pay2011} (see also \cite[Theorem 10]{HMvN}), and a short proof of Magnus' Theorem
(Sec.\ref{ss.expMagnus}).  Finally we will give a natural and straightforward
proof of the Connes-Moscovici Rearrangement Lemma in the context of the
multivariate holomorphic functional calculus.

The paper is organized as follows. Section \ref{s.MHFC} reviews the mulivariate
holomorphic functional calculus for a commuting tuple in a commutative Banach
algebra and our ``na\"ive'' extension for commuting tuples in a general Banach
algebra.  Section \ref{s.DDmultifc} studies divided differences of a holomorphic
functions with the methods of the na\"ive holomorphic functional calculus
framework. In Section \ref{s.Applications} we present the various applications
mentioned before.  Finally in the Appendix \ref{s.Appendix} we collect various
combinatorial formulas about divided differences which are needed in the
applications section.

\subsection*{Acknowledgement} This project was begun after \cite{Les2014} was
finished and has long been stalled. Over the years we have benefited from
numerous discussions with our mathematical friends. It is impossible to mention
all of them. Explicitly, we would like to mention our long term collaborator
Boris Vertman whose mathematical enthusiasm is particularly inspiring in times
of frustration. The second named author would also like to thank Elmar Schrohe,
Hermann Schulz-Baldes, and Achim Klenke for invitations to give talks about the
project. We would like to thank the anonymous referee for valuable suggestions,
in particular we owe the important reference \cite{Qui1985} to her/him.

\section{Multivariate holomorphic functional calculus}
\label{s.MHFC}
\subsection{Notation}
\label{ss.notationMultiFuncCal}

We fix some notation which will be used frequently. $\N$ denotes the set of
natural numbers including $0$.  Multiindices will usually be denoted by greek
letters. Sums of the form $\sum_\ga$ will be over $\ga\in\N^n$ with further
restrictions indicated below the $\sum$ sign.

During this section script letters $\sA, \sB,\ldots$ denote \emph{unital} Banach
algebras. $\sAn_\pi$ denotes the projective tensor product completion of
$\sA\otidots\sA$ ($n$ factors), \textit{cf., e.g.,} \cite{Gel1959}. That is
$\sAn_\pi$ is the completion of $\sAn$ with respect to the norm
\[
\| x \|_\pi = \inf \sum_i \| a_1^{(i)}\|\cldots \| a_{n}^{(i)}\|,
\]
where the infimum is taken over all representations of $x\in\sAn$ as a finite
sum $\sum_i a_1^{(i)}\otidots a_n^{(i)}$.  $\sAn_\pi$ is a Banach algebra.

Given a subset $U\subset \C^n$ we denote by $\sO(U)$ the Fr\'echet algebra of
holomorphic functions on $U$ with the topology of uniform convergence on compact
subsets.

\subsection{Review of the holomorphic functional calculus } 
\label{ss.RHF}

The holomorphic functional calculus in one variable can already be found in
Gelfand's fundamental paper on normed rings \cite[p. 20]{Gel1941}: given an
element $a$ in a Banach algebra $\sA$ there is a unique continuous homomorphism
$\theta_a$ from the Fr\'echet algebra $\sO(U)$ into $\sA$ sending the function
$\id_U:z\mapsto z$ to $a$.

The generalization to $f\avec n$ for several commuting Banach algebra elements
and a holomorphic function $f$ of $n$ variables is surprisingly complex and from
Gelfand's paper it took almost 30 years until the theory reached a mature state.
More precisely, let $a=\avec n$ be $n$ commuting elements in $\sA$ and let
$U\subset \C^n$ be an open neighborhood  of the \emph{joint spectrum}
$\specA\avec n$ (to be discussed below) one is aiming for a continuous
\emph{algebra homomorphism} $\theta:\sO(U)\to \C$ sending the coordinate
function $z_j$ to $a_j$.

This aim is obstructed by various difficulties. Firstly, in a general Banach
algebra there is no general accepted notion of joint spectrum for a commuting
tuple $a=\avec n$. One therefore has to be more modest and embed $\alist n$ into
a commutative Banach algebra where the joint spectrum can be defined using ideal
theory; this will be explained below.  For commuting operators on a Banach space
there is a notion of joint spectrum which is independent of the choice of an
ambient commutative Banach algebra \cite{Tay1970-JFA} as well as a holomorphic
functional calculus based on an algebraic approach to the Cauchy-Weil integral
\cite{Tay1970}.  This celebrated approach, however, comes with the price of
being dauntingly complex.

We will be more modest here and briefly review the calculus if $\alist n$ are
elements of a \emph{commutative} Banach algebra $\sA$.  Historically, there
exist basically two approaches. One using polynomial approximation and the
Oka-Weil approximation theorem (see
\cite{All1969,All1971,Wael1954,Hor1990,Sto1971}, for a modern textbook treatment
see \cite[Part III]{All2011}), and the second one being closer to Gelfand's
Cauchy integral by using a direct integral description of the functional
calculus (see \cite{Shi1955,ArCa1955,Bou1967,Gam1969}).  Both approaches have in
common, that the uniqueness statement is more involved\footnote{%
See also Zame \cite{Zam1979} whose uniqueness statement is stronger than what we
quote below. On the other hand even Zame's statement is not a plain extension of
Gelfand's clear and straightforward one variable statement as he needs to
\emph{assume} a version of the spectral mapping theorem. The latter is a
consequence in $1$D. However, the statement of Zame is for one fixed tuple
$a = \avec n$, see \cite[Theorem 3.5]{Zam1979}, while the uniqueness statement
we quote talks about \emph{all} tuples at once.}.
\label{Zamefoot}
We will use the latter approach, following mostly the excellent
monograph \cite[Chap. III]{Gam1969}.

So therefore from now on we assume that the Banach algebra $\sA$ is commutative
and unital. Let $M_{\sA}$ be the maximal ideal space of $\sA$, realized as the
set of continuous unital algebra homomorphisms $\phi:\sA \to \C$. This is a
compact Hausdorff space equipped with the weak-*-topology as a subset of the
unit ball in the dual $\sA'$; $\phi\leftrightarrow \ker \phi$ is the
identification between $\phi$ and the maximal ideal $\ker\phi\subset \sA$. Given
$a:=\avec n\in \sA^n$, the \emph{joint spectrum} of $a$ is denoted by
$\specA(a) = \specA\avec n$ and defined as the set
\begin{equation}\label{eq:RHF.1}
\specA(a)
    :=\bigsetdef{ ( \phi(a_1),\ldots,\phi(a_n) )\in \C^n}{%
        \phi\in M_{\sA}}.
\end{equation}
The joint spectrum is a nonempty compact subset of $\C^n$, and if $\sA$ is
finitely generated then $\specA(a)$ is \emph{polynomially convex}, cf.
\cite[Sec. III.1]{Gam1969}. Moreover, $z=\gvec z \not\in \specA\avec n$ is
equivalent to the existence of $\glist g \in \sA$ such that
\begin{equation}\label{eq:RHF.2}
\sum_{j=1}^{n} (z_j-a_j)\cdot g_j = 1_{\sA}.
\end{equation}
It is worth noting that the joint spectrum depends on the algebra, as demonstrated by the following example.
\begin{example}\label{ex:1}
Let $\Omega_{r} = \setdef{z\in \C}{1\leq |z| \leq r}$, where $1\leq r $. For
$r>1$, let $\sA_r$ be the algebra of continuous functions on $\Omega_r$ that are
holomorphic in the interior of $\Omega_r$; for $r=1$ let $\sA_1 = C(S^1)$ be the
algebra of continuous functions on the circle of radius one.  Using the maximum
principle one finds that for $f\in\sA_r$ the norm is given by
\begin{equation}\label{eq:ex1.1}
   \|f\|_{\sA_r} = \max\bl \max_{|z|=1} |f(z)|,
                      \max_{|z|=r} |f(z)| \br.
\end{equation}
Denote by $a_r\in\sA_r$ the function $z\mapsto z$ viewed as an element of
$\sA_r$.  Then one checks that $\spec_{\sA_r}(a_r) = \Omega_r $, for all $r$.
Moreover, if $r<r'$ then $\sA_{r'} \subset \sA_{r}$ is naturally a subalgebra
and under this inclusion $a_{r'}$ is identified with $a_r$. Furthermore,
$\spec_{\sA_r}(a_r)\subsetneqq \spec_{\sA_{r'}}(a_{r'})$.

When $r=1$ then $\sA_1$ is a $C^\ast$-algebra and therefore the spectrum is
stable, \ie for any $C^\ast$-algebra
$\sB \supset \sA_1$, $\spec_{\sB}(a_1) = \spec_{\sA_1}(a_1)$.

On the other hand, consider the Banach subalgebra $\sC_r\subset \sA_r$ generated
by $a_r$. That is any $f\in\sC_r$ is, as a function on $\Omega_r$, the uniform
limit of a sequence of polynomials $p_n(a_r)$ in $a_r$. Hence, by the maximum
principle, the polynomials $p_n(a_r)$ converge uniformly on the full ball
$B_r:=\bigsetdef{z\in\C}{|z|\le r}$ and hence for $f\in\sC_r$ we have, cf.
\eqref{eq:ex1.1}
\begin{equation}\label{eq:ex1.2}
   \|f\|_{\sC_r} = \max_{|z|=r} |f(z)|
         = \max_{|z|\le r} |f(z)|,
\end{equation}
and thus we have, again by the maximum principle, natural isometric inclusions
\begin{equation}
     \sC_r\hookrightarrow C(B_r),
     \sC_r\hookrightarrow \sA_r \hookrightarrow C(\Omega_r)
\end{equation}
and the continuous inclusion $\sC_{r'}\subset \sC_{r}, r<r'$.

The spectrum $\spec_{\sC_r}(a_r)$ is equal (!) to
$B_r\supsetneqq \spec_{\sA_r}(a_r)=\Omega_r$.
\end{example}

The holomorphic functional calculus, as stated in \cite[Theorem III.4.1]{Gam1969},
is summarized in the next theorem.

\begin{theorem}\label{theo.gamelin}
Let $\sA$ be a unital \emph{commutative} Banach algebra and let $\specA(a)$ be
the joint spectrum of $a:=(a_1,\ldots,a_n)\in \sA^n$, for some $n\in \N^\ast$.
For an open set $U\subset\C^n$ let $\sO(U)$ be the Fr{\'e}chet algebra of
holomorphic functions on $U$ equipped with the topology of compact($=$locally
uniform) convergence.

Then there is a unique \emph{family} of continuous linear maps\footnote{%
The uniqueness statement here is a little involved: the family is parametrized
by the $a\in\sA^n$, $U\supset\specA(a)$ open, and uniqueness is guaranteed only
for the whole family.  This was later improved by Zame \cite[Theorem
3.5]{Zam1979} whose uniqueness statement refers only to one tuple $a$.  See,
however the footnote on page \pageref{Zamefoot}.}
\begin{equation}\label{DDNote-RG2.1}
	\Theta_a:=\Theta_a^U:\sO(U) \to \sA, \quad a\in\bigcup_{n=1}^\infty \sA^n,
        \; U\supset \specA(a), U\text{ open}
\end{equation}
such that:

\textup{1.} For a polynomial, $F\in \C[z_1,\ldots,z_n]$,
of the form
\[
F(z)= F(z_1,\ldots,z_n )
            =\sum\limits_{\alpha\in \N^n}
               \lambda_\alpha\cdot z_1^{\alpha_1}\cdots z_n^{\alpha_n}
            =: \sum\limits_{\ga\in\N^n} \gl_\ga \cdot z^\ga,
\]
we have
\[
	\Theta_a(F) =: F(a) = F(a_1,\ldots,a_n)
       = \sum_{\alpha\in \N^n} \lambda_\alpha\cdot a_1^{\alpha_1}\cdots a_n^{\alpha_n}
       =: \sum\limits_{\ga\in\N^n} \gl_\ga \cdot a^\ga.
\]

\textup{2.} For $F\in \sO(U)$, if $a_{n+1},\ldots,a_m\in \sA$ and if $\tilde F$
is the trivial extension of $F$ to $U\times \C^{m-n}$ defined by
$\tilde F(z_1,\ldots,z_m) = F(z_1,\ldots,z_n)$ then, denoting
$\tilde{a}:=(a_1,\ldots,a_m)$,
\[
    \Theta_{\tilde a}(\tilde F) = \tilde F(\tilde a) = F(a) = \Theta_a(F).
\]

\textup{3.} If $\{F_k\}_{k\in \N}\subset \sO(U)$ is a sequence which converges
locally uniformly to $F\in \sO(U)$ then
$\Theta_a(F_k) \rightarrow \Theta_a(F)$.
\end{theorem}

\begin{remark}\label{rem_2.3}

1. It then \emph{follows} that $\Theta_\bullet$ is natural in all respects
conceivable (\cite{Gam1969}, after Theorem III.4.1), in particular, $\Theta_a$
is a homomorphism of algebras and the Gelfand transform
$\widehat{F(a_1,\ldots,a_n)}$ of $F(a_1,\ldots,a_n)$ equals
$F(\hat{a}_1,\ldots,\hat{a}_n)$, \ie the Gelfand transforms of
$a_1,\ldots, a_n$ inserted into $F$.

2. We emphasize furthermore, that if $U\subset V$ and $F\in\sO(V)$ then
$\Theta_a^V(F)=\Theta_a^U(F_{|U})$.  Also by construction it follows that if
$\sA\subset\sB$ is contained in a larger commutative Banach algebra then for
$a\in\sA^n$ the value of $\Theta_a(F)$ will be the same regardless of whether it
is taken w.r.t. to $\sA$ or w.r.t. $\sB$.
\end{remark}

\newcommand\thmGamelin{\ref{theo.gamelin}}
\begin{proof}[Idea of Proof]
The main technical step in the proof is to show that if $w\in C^\infty_c(U;\sA)$
is such that $w\equiv1$ in a neighborhood of $\spec_\sA(a)$ then there are smooth
$\sA$-valued functions $u_j\in C^\infty(\C^n;\sA)$, $j=1,\ldots, n$, such that
\[
\sum_{j=1}^n u_j(z)(z-a_j) = 1 - w(z).
\]
It then follows that the differential form
$du_1\wedge dz_1\wedge\ldots \wedge du_n\wedge dz_n$
is supported on $U$. For $F\in \sO(U)$ one then puts
\begin{equation}\label{DDNote-RG2.2}
\Theta_a(F):= F(a) = F(a_1,\ldots,a_n)= n!\int_U F du_1 \dsl z_1 \cdots du_n
\dsl z_n,
\end{equation}
where $\dsl z_j = \frac{1}{2\pi i} d z_j$.
The integral is indeed independent of the choices ($u, w$) made\footnote{
It is worth nothing that for $\sA=\C^n, a=(a_1,\ldots, a_n)\in\C^n$
\Eqref{DDNote-RG2.2} is the Cauchy-Weil formula for the holomorphic function $F$ \cite{Wei1935}.}
Uniqueness then follows by invoking the Oka-Weil approximation Theorem
\cite[Theorem III.5.1]{Gam1969} and the Arens-Calder\'on lemma
\cite[Theorem III.5.2]{Gam1969}.

We add that the claims of Remark \ref{rem_2.3} 2. about the dependence on $U$
resp. on the algebra $\sA$ are an immediate consequence of the formula
\Eqref{DDNote-RG2.2} and its independence on the choices made.
\end{proof}

\subsection{Na\"ive holomorphic functional calculus}
\label{ss.NHFC}

We will need a holomorphic functional calculus for several commuting elements in
a \emph{noncommutative} Banach algebra $\sA$.  The most na\"ive approach that
comes to mind would be the following: given commuting elements $\alist n\in\sA$
of a not necessarily commutative Banach algebra $\sA$. Consider the Banach
algebra $\sB\subset \sA$ generated by $\alist n$ and use the functional calculus
in $\sB$. However, this has several disadvantages: firstly, as the Example
\ref{ex:1} shows, the joint spectrum of $\avec n$ in $\sB$ might be larger than
the joint spectrum in $\sA$. Secondly, the Banach subalgebra $\sB$ depends on
the tuple.  Apart from that, the formula for the $\Theta_a$ in Theorem
\ref{theo.gamelin} above is not very explicit, not to speak of the difficulties
of the Taylor approach to the functional calculus, see Section \ref{ss.RHF}
above.

We propose therefore an elementary na\"ive approach very close to the one
outlined before but with the advantage that firstly the joint spectrum does not
become too large, secondly the calculus does not depend on the choice of a
specific commutative subalgebra and thirdly, and maybe most importantly, that
there is a concrete contour integral formula similarly to the one in the one
variable case.


Therefore, let $\sA$ be a unital Banach algebra (not necessarily commutative)
and let $a=\avec n\in\sA^n$ be commuting elements. We put
\begin{equation} \label{eq:NHFC.1}  
  \specn(a):= \prod_{j=1}^n \spec_{\sA}(a_j);
\end{equation}
the subscript \enquote{na} indicates that this is a \enquote{na\"ive} definition
of joint spectrum.  The previous notion of joint spectrum leads to a potentially
smaller subset of $\specn(a)$.

However, with this notion of joint spectrum we have the following result:


\begin{theorem}[Na\"ive holomorphic functional calculus]
\label{thm:NHFC}
Let $\sA$ be a unital (noncommutative) Banach algebra and let
$a = \avec n\in \sA^n$ be  an $n$-tuple of commuting elements.  Furthermore, let
$U_j\supset \specA(a_j)$, $1\le j\le n$, be an open neighborhood and put
$U = \prod_{j=1}^{n} U_j$.

If $\sB\subset \sA$ is any commutative unital Banach algebra containing
$\alist n$ \emph{and} their resolvents
$(\lambda-a_j)^{-1}$, $\lambda \in \C\setminus\specA(a_j)$, then in the
commutative Banach algebra $\sB$ we have
\[
\spec_{\sB}(a)\subset\specn(a).
\]
Furthermore, the map $\Theta_a^\sB$ of Theorem \ref{theo.gamelin} \footnote{%
The superscript $\sB$ in $\Theta_a^\sB$ indicates that a priori the functional
calculus of Theorem \ref{theo.gamelin} depends on the ambient algebra $\sB$.
}
is given by
\begin{equation}\label{eq:NHFC.2}    
 \Theta_a^\sB(f) = f(a) = f\avec n
 := \int_{\gamma_1} \cdots \int_{\gamma_n}
      f(z)\cdot \prod_{j=1}^{n}(z_j-a_j)^{-1} \dsl z,
\end{equation}
where $\dsl z = \dsl z_1\cdots \dsl z_n$,
$\dsl z_j = \frac{1}{2\pi i} d z_j$,
and $\gamma_j, 1\le j\le n,$ are integration cycles\footnote{An
integration cycle is a closed chain of rectifiable paths.} in
$U_j\setminus \specA(a_j)$ which encircle $\specA(a_j)$
once, resp.

Additionally, denoting by $\sA^n_U$ the set of those $a\in \sA^n$ with
$\specn(a)\subset U$ the map $\cN: \sA_U^n \times  \sO(U) \to \sA$ defined by
the right hand side of \eqref{eq:NHFC.2}
\begin{equation}\label{eq:NHFC.3}
(a,f)\mapsto \cN_a(f) := \int_{\gamma_1} \cdots
\int_{\gamma_n} f(z)\cdot \prod_{j=1}^{n}(z_j-a_j)^{-1} \dsl z = f(a),
\end{equation}
is continuous, where $\sA^n_U$ is an open set of $\sA^n$ and $\sO(U)$ is the Fréchet algebra of holomorphic functions on $U$ (as in Theorem \ref{theo.gamelin}) .
\end{theorem}
\begin{remark}
1. One might ask, why not, in light of the discussion at the
beginning of this Subsection \ref{ss.NHFC}, just take $\sB$
to be the Banach algebra $\tilde\sB$ \emph{generated} by $\alist n$
\emph{and} their resolvents. The point we make here, however,
is that the result will be the same in \emph{any}
commutative Banach algebra $\sB$ with $\tilde\sB\subset\sB\subset\sA$.

2. Our continuity statement \eqref{eq:NHFC.3} goes beyond the continuity
statement of Theorem \thmGamelin\ which is for a fixed tuple $a$ only. Also
notice, that for $a,a'\in \sA^n_U$ ($a$ and $a'$ might \emph{not} commute) the
corresponding algebras $\sB, \sB'$ might differ and hence the map $\cN$ cannot
easily be expressed in terms of the map $\Theta_\bullet^\bullet$ of Theorem
\thmGamelin. One should view the RHS of \eqref{eq:NHFC.3} as the
\emph{definition} of the map $\cN$. \eqref{eq:NHFC.2} then says that for given
$a$ and its corresponding algebra $\sB$ one has $\Theta_a^\sB(f) = \cN_a(f)$.

3. We emphasize that the algebra $\sA$ is not necessarily commutative and that
the map $\cN$ does not depend on any choices.  Furthermore, for fixed $a$ it
enjoys all the properties of the holomorphic functional calculus of Theorem
\thmGamelin\ and it is given by a simple contour integral. The obvious
disadvantage is that the na\"ive joint spectrum is not as small as possible and
that the na\"ive calculus is restricted to functions being holomorphic in a
neighborhood of the na\"ive joint spectrum.
\end{remark}

From now on we will wherever unambiguously possible write the more suggestive
$f(a)$ instead of $\cN_a(f)$ or $\Theta_a^\sB(f)$.

\begin{proof} First fix $a\in \sA^n$ such that $a=\avec n$ is an $n$-tuple of
commuting elements.

1.  Then for $b\in\{\alist n\}$ and $\gl\not \in\specA(b)$ one has by
construction of the algebra $\sB$ that $\gl\not\in\spec_\sB(b)$. Let $f(z):=
\gl - z$ and $g(z):= (\gl - z)\ii$. Then Theorem \thmGamelin, 1. implies
$\Theta_b^\sB(f) = \gl - b$ and since $f\cdot g = 1$ and since
$\Theta_b^\sB$ is a homomorphism of algebras, it also follows that
$\Theta_b^\sB(g) = (\gl - b)\ii$.

2. Next we show that the joint spectrum of $a$ in $\sB$, $\spec_{\sB}(a)$, is
contained in $\specn(a)$. Namely, let $\gl = \gvec \gl\not\in\specn(a)$. Then
for at least one $j\in\{1,\ldots,n\}$ we have $\gl_j\not\in\spec_\sB(a_j)$.
Then by 1. of this proof $g_j := (\gl_j - a_j)\ii\in\sB$ and hence
$(\gl_j-a_j)\cdot g_j = 1_\sB$, showing that $\gl\not\in\spec_\sB(a)$, cf.
\Eqref{eq:RHF.2}.

3. Still for fixed $a$ now choose $U_j$ and the integration cycles $\gamma_j$ as
in the statement of the theorem. The right hand side of \eqref{eq:NHFC.2} is a
limit of Riemann sums of the form
\[
  \sum f(\gamma_1(t_{k_1}),\ldots,\gamma_n(t_{k_n}))
  \prod_{j=1}^n (2\pi i)\ii\cdot (\gamma_j(t_{k_j}) - a_j)\ii
         \bl \gamma_j(t_{k_j}) - \gamma_j(t_{k_j-1}) \br.
\]
It follows from Theorem \thmGamelin\ and 1. of this proof that this Riemann sum
is $\Theta_a^\sB$ applied to the holomorphic (rational) function
\[
 z\mapsto  \sum f(\gamma_1(t_{k_1}),\ldots,\gamma_n(t_{k_n}))
  \prod_{j=1}^n (2\pi i)\ii\cdot (\gamma_j(t_{k_j}) - z_j)\ii
         \bl \gamma_j(t_{k_j}) - \gamma_j(t_{k_j-1}) \br.
\]
In a neighborhood of $\spec_\sB(a)$ the approximating Riemann sums therefore
converge uniformly to the holomorphic function $f$  and by the continuity statement
of Theorem \thmGamelin\ the equality \eqref{eq:NHFC.2} follows.

4. It remains to prove the continuity of the map $\cN$.  We first note that
clearly the set $\sA_U^n$ is an open subset of $\sA^n$.  So given $a\in\sA_U^n$
choose $U_j$ and $\gamma_j$ as above.  Furthermore, choose
$V_j\supset\spec(a_j)$ open with $\ovl{V_j}$ contained in the interior\footnote{
Of course, ``interior'' is meant here in the sense of integration cycles, \ie
the region encircled by the $\gamma_j$.}
of $\gamma_j$ and put $V:=\prod_j V_j$. Then $\sA_V^n$ is an open neighborhood of
$a$ and for $b\in\sA_V^n$ and $g,f\in\sO(U)$ we have the estimates
\begin{align}
  \|f(b)-g(a)\| &\le \|f(b)-f(a)\|+ \|f(a)-g(a)\|,  \\
  \|f(a)-g(a)\| &\le \|f-g\|_{\infty,\gamma}\cdot
\int_{\gamma_1\times\ldots\times \gamma_n} \Bigl\|\prod_j (z-a_j)\ii \Bigr\|
\,|\dsl z|,\label{DDNote-G2.4}\\
\|f(b)-f(a)\| &\le \|f\|_{\infty,\gamma} \cdot
\int_{\gamma_1\times\ldots\times \gamma_n} \bigl\|\prod_j (z-a_j)\ii - \prod_j
(z-b_j)\ii\bigr\| \,|\dsl z|,
\end{align}
where $\|\cdot\|_{\infty,\gamma}$ stands for the $\sup$ norm on the image of the
curves $\gamma_1,\ldots,\gamma_n$. These estimates imply the continuity of the
map in \Eqref{eq:NHFC.3}.
\end{proof}

We record here the (expected) behavior of the na\"ive calculus on rational
functions and on tensor products:

\begin{prop}\label{DDNote-P2.1}
Under the hypothesis of Theorem \ref{thm:NHFC}, let $\gamma_j, U_j$ be as above,
$j=1,\ldots,n$.  If $f_j\in\sO(U_j)$, $1\leq j \leq n$, then
\[
    (f_1\otidots f_n)(a) = f_1(a_1)\cldots f_n(a_n).
\]
In particular, if $p(z)= \MultiSum \ga n p_\ga z^\ga$ is a polynomial (or an
entire function) then
\[
  p(a) = \MultiSum \ga n p_\ga \cdot a_1^{\ga_1}\cldots a_n^{\ga_n} = \sum_{\alpha\in \N^n} p_\alpha \cdot a^\alpha.
\]
Moreover, for the resolvent $R_z(\gl):= \prod\limits_{j=1}^n (z_j-\gl_j)\ii$
we have $R_z(a)=\prod\limits_{j=1}^n (z_j-a_j)\ii$.
\end{prop}
\begin{proof} This is immediate from \Eqref{eq:NHFC.2}
\end{proof}


\section{Divided Differences as multivariate holomorphic functional calculus}
\label{s.DDmultifc}

We now establish the connection between divided differences and multivariate
holomorphic functional calculus. To facilitate this connection, a concise
overview of divided differences, including pertinent results necessary for our
discussion, is provided in Appendix \ref{s.Appendix}.

During this section $\sA$ denotes a unital Banach algebra.

\subsection{Divided Differences of commuting elements} 
\label{ss.DDcomutelements}

Let $U\subset\C$ be an open set and $f\in\sO(U)$. Then the $n$-th divided
difference
\[
  \DD^nf(z):=[z_0,\ldots,z_n]f
\]
is a holomorphic function of $z=(z_0,\ldots,z_n)\in U^{n+1}$.

Now let $a=(a_0,\ldots,a_n)\in\sA^{n+1}$ be commuting elements and let
$U\supset \cup_{j=0}^n \spec a_j$ be an open neighborhood of the union of the
spectra of $a_0,\ldots, a_n$. Furthermore, if $f\in\sO(U)$ then the $n$-th
divided difference $\DD^nf$ is holomorphic in $U^{n+1}$ which is an open
neighborhood of $\specn(a)$.  Thus
\[
  \DD^nf(a)=[a_0,\ldots,a_n]f\in\sA
\]
is well-defined by the na{\"i}ve holomorphic functional calculus established in
Section \ref{ss.NHFC}.

Choosing an integration cycle $\gG$ in $U$ which encircles
$\cup_{j=1}^n\spec(a_j)$ exactly once, we have in the interior,
$\operatorname{int}\gG$, of $\gG$ the integral representation
\[
  f(\cdot) = \int_\gG f(\gz) (\gz-\cdot)\ii \, \dsl \gz,
\]
where the integral converges in the Fr\'echet space $\sO(\operatorname{int}
\gG)$. Thus by the continuity statement of Theorem \ref{thm:NHFC}
\begin{equation} \label{DDNote-G2.6}
  \DD^nf =
  \int_\gG f(\gz) \cdot \DD^n\Bl(\gz-\cdot)\ii \Br\, \dsl \gz.
\end{equation}
In view of the continuity properties of the na\"ive holomorphic functional
calculus of Theorem \ref{thm:NHFC} and \Eqref{EqDDResolvent} we therefore find
\begin{equation} \label{DDNote-G2.7}
 \DD^n f (a) =  \DDots a_0,a_n; f
     =\int_\gG f(\gz)\cdot \prod_{j=0}^n (\gz-a_j)\ii\,\dsl \gz.
\end{equation}
Similarly, the Genocchi-Hermite formula, cf. \refDDGenocchi,
\begin{equation}\label{Eq-Genocchi-Hermite}
  \DDots a_0,a_n; f = \ISimplex n f^{(n)}\bl\sum_j s_j a_j\br \, ds
\end{equation}
holds for $a_0,\ldots,a_n\in\sA$ whenever $f$ is holomorphic in a neighborhood
of
\[
  \ovl{\bigcup_{s\in\Simplex_n} \spec \Bl \sum_j s_j a_j \Br},
\]
where $\Simplex_n$ denotes the $n$-standard simplex.  Also, if $f$ is
holomorphic in the ball $\bigsetdef{z\in\C}{|z|<r}$ and if $a_0,\ldots, a_n$ are
commuting elements in $\sA$ of norm less than $r$ then by Proposition
\ref{DDNote-P1.3} and Proposition \ref{DDNote-P2.1}
\[
  \DDots a_0,a_n; f = \MultiSum \ga {n+1}
  \frac{f^{(n+|\ga|)}(0)}{(|\ga|+n)!}\; a_0^{\ga_0}\cldots a_n^{\ga_n}
\]
with absolute convergence in norm.

\subsection{Divided differences of a general tuple } 
\label{ss.DDgeneralelementproj}

We now consider arbitrary not necessarily commuting elements
$a_0,a_1,\ldots\in \sA$. Recall from Section \ref{ss.notationMultiFuncCal} the
projective tensor product $\sAn_\pi$.

The multiplication map
\begin{equation} \label{DDNote-G2.8}
  \mu_n: \sAn\longrightarrow \sA,\; a_1\otidots a_n
  \mapsto a_1\cldots a_n
\end{equation}
extends by continuity to a bounded linear map from the projective tensor product
$\sAn_\pi$ into $\sA$.

For $a\in\sAn[n+1]_\pi$ and $b_1,\ldots, b_n\in\sA$ we write
\begin{equation} \label{DDNote-G2.9}
 a(b_1\cldots b_n) :=\mu_{n+1} \bl a\cdot(b_1\otidots b_n\oti \oneA)\br \in\sA.
\end{equation}
Note that
\begin{equation} \label{DDNote-G2.10}
  (a_0\otidots a_n)(b_1\cldots b_n)= a_0 \cdot b_1\cdot a_1\cldots b_n \cdot a_n.
\end{equation}
\Eqref{DDNote-G2.9} defines a continuous bilinear map
$\sAn[n+1]_\pi\times \sAn_\pi\to\sA$.  For $a\in\sA$ we introduce, \cf
\cite[Sec.~3.2]{Les2014}
\begin{equation}
  \begin{split} \label{DDNote-G2.11}
  a^{(j)} & := \oneA\otidots \oneA\oti a\oti \oneA\otidots \oneA, \;\; 0\le
  j\le n  \;\; \text{($a$ is in the $j$--th slot)}, \\
  \nabla_a^{(j)} &:= a^{(j-1)} - a^{(j)}, \quad 1\le j \le n.
  \end{split}
\end{equation}
Note that $a^{(j)}, \nabla_a\pup j;$ also depend on $n$, we suppress this from
the notation. Note that we use a different sign convention as \cite{Les2014}.
Our $\nabla_a$ here corresponds to $-\nabla_a$ in \cite{Les2014}.  For future
reference we note
\begin{equation} \begin{split}
   \nabla_a\pup j;+\ldots+\nabla_a\pup n;+a\pup n; &=a\pup j-1;\\
   a\pup0;-\nabla_a\pup 1;-\ldots-\nabla_a\pup j;  &=a\pup j;
     \end{split},\qquad 1\le j\le n.
   \label{DDNote-G2.12}
\end{equation}

\begin{example} We give an application. If $n=1$ then
$\nabla_a:=\nabla_a\pup 1;$ and the pairing with $\nabla_a$ gives the adjoint
action of $a$. Namely,
\begin{equation} \label{DDNote-G2.12-1}
  \nabla_a(b):= (a\oti \oneA - \oneA\oti a)(b)
      = a\cdot b \cdot\oneA - \oneA\cdot b\cdot a = \ad_a(b),
\end{equation}
and inductively
\begin{equation} \label{eq_adjoint}
    \nabla_a^n(b) = \ad_a^n(b).
\end{equation}
Note that $\nabla_a^n(b)$ on the left stands for the pairing
\Eqref{DDNote-G2.9}
between $\bl\nabla_a\pup 1;\br^n\in\sAn[2]_\pi$ and $b\in\sAn[1]_\pi$.

The pairing \Eqref{DDNote-G2.9} together with \Eqref{DDNote-G2.10}
is of some use for giving
proofs of combinatorial identities. E.g.
\[\begin{split}
    a^m b &= \bl a\pup 0;\br^m(b)=\bl\nabla_a + a\pup 1;\br ^m(b)
           = \sum_{j=0}^m {m\choose j} \Bl \nabla_a^j \bl a
              \pup 1;\br^{m-j}\Br (b)\\
          & = \sum_{j=0}^m {m\choose j} \cdot \ad_a^j(b)\cdot a^{m-j},
\end{split}\]
since $\nabla_a$ and $a\pup 1;$ commute.

This can be extended to the multinomial case. Namely, for arbitrary $n\in\N$
\[\begin{split}
    a^m\cdot b_1\cldots b_n
    &= \Bl \nabla_a\pup1;+\ldots+\nabla_a\pup n;+a\pup n;\Br^m(
         b_1\cldots b_n)\\
         &= \sum_{\substack{\alpha\in \N^{n}  \\ |\ga|\le m}} \frac{m!}{\ga! (m-|\ga|)!}\;
           \ad_a^{\ga_1}(b_1)\cldots
           \ad_a^{\ga_n}(b_n) \cdot a^{m-|\ga|}. \qed
           \end{split}\]
\end{example}

\subsubsection{Divided differences of a general tuple} 

After these preparations we are now able to define the divided difference of a
general tuple as an element of the projective tensor product.

For arbitrary $a_0,\ldots,a_n\in\sA$ the elements $a_i\pup
i;\in\sAn[n+1]_\pi, 0\le i\le n$, commute. By slight abuse
of notation we put
\begin{equation}
  \begin{split} \label{DDNote-G2.13}
\DDotsg a_0,a_n; f&:=
  \DDots {a_0\pup 0;},{a_n\pup n;}; f\in\sAn[n+1]_\pi\\
    &= \int_\gG f(\gz)\cdot (\gz-a_0)\ii\oti\ldots \oti (\gz-a_n)\ii
  \,\dsl\gz.
  \end{split}
\end{equation}
Thus, for $a_0,\ldots,a_n\in\sA, b_1,\ldots, b_n\in\sA$ and
appropriate $f\in\sO(U)$ (see the beginning of this section ) one obtains, in view of \Eqref{DDNote-G2.7},
\begin{equation}
  \begin{split}
\DDotsg a_0,a_n; &f(b_1\cldots b_n)\\
    &=
\int_\gG f(\gz) \cdot (\gz-a_0)\ii \cdot b_1\cdot (\gz-a_1)\ii\cldots b_n\cdot
   (\gz-a_n)\ii \,\dsl\gz,
   \label{DDNote-G2.14}
  \end{split}
\end{equation}
where this integral is well defined by Theorem \ref{thm:NHFC}.

\section{Applications}
\label{s.Applications}

We now direct our focus towards demonstrating the practical implications of the
constructions developed earlier by presenting a variety of applications. As in the previous section, upper case script letters $\sA, \sB,\ldots$ denote
\emph{unital} Banach algebras.


\subsection{The \nonc\  Newton interpolation formula} 
\label{ss.noncommutativeNewton}

\begin{prop} \label{DDNote-P2.3} Let $\sA$ be a unital Banach algebra and let
$f\in\sO(U)$ be a holomorphic function in the open set $U\subset \C$. Then for
$a_0,\ldots, a_n\in\sA$ with $\spec(a_j)\subset U, 0\le j\le n,$ we have
\[
 f(a_n) =f(a_0)+
    \sum_{j=1}^n \bl \DDotsg a_0,a_j; f\br \bl (a_n-a_0)\cldots
    (a_n-a_{j-1})\br.
\]
\end{prop}
This is a \nonc\  analogue of Newton's interpolation formula.  We emphasize that
the $a_0,\ldots,a_n$ do not necessarily commute.  Therefore, the order of the
entries on the right matters.

For the proof we first establish a recursion formula for \nonc\  divided
differences.

\begin{lemma} \label{DDNote-L2.4} Under the assumptions of Proposition
\ref{DDNote-P2.3} we have
\[\begin{split}
\DDotsg a_0,{a_{n-1},a_{n+1}}; &f(b_1\cldots b_n) -
    \DDotsg a_0,a_n; f(b_1\cldots b_n) \\
  &=
\DDotsg a_0,{a_{n},a_{n+1}}; f\bl b_1\cldots b_n\cdot(a_{n+1}-a_n)\br .
\end{split}\]
\end{lemma}
\begin{proof} This is a consequence of the resolvent equation. With regard to
\Eqref{DDNote-G2.6}, \eqref{DDNote-G2.13}, \eqref{DDNote-G2.14}, and
\eqref{EqDDResolvent} we note that
\begin{multline*}
\DDotsg a_0,{a_{n-1},a_{n+1}}; (\gz-\cdot)\ii (b_1\cldots b_n) -
              \DDotsg a_0,a_n; (\gz-\cdot)\ii (b_1\cldots b_n) \\
    =  (\gz-a_0)\ii\cdot b_1\cldots (\gz-a_{n-1})\ii\cdot b_n
     \cdot \Bigl\{ (\gz-a_{n+1})\ii - (\gz - a_n ) \ii \Bigr\}.
\end{multline*}
By the resolvent equation the expression in curly braces equals
\[
   (\gz - a_n)\ii\cdot (a_{n+1}-a_n)\cdot (\gz -a_{n+1})\ii.
\]
Integration over the contour in \Eqref{DDNote-G2.14} then proves the claim.
\end{proof}

\begin{proof}[Proof of Proposition \plref{DDNote-P2.3}]
We proceed by induction on $n$. For $n=0$ the claim is obvious.  So assume it
holds for $n$. To prove it for $n+1$ we apply it to $a_0,\ldots,a_{n-1},a_{n+1}$
and find
\begin{align*}
 f(a_{n+1}) & = f(a_0)+
 \sum_{j=1}^{n-1} \bl \DDotsg a_0,a_j; f\br \bl (a_{n+1}-a_0)\cldots
 (a_{n+1}-a_{j-1})\br \\
 & \quad + \bl[a_0,\ldots,a_{n-1},a_{n+1}]_\pi f\br
     \bl (a_{n+1}-a_0)\cldots (a_{n+1}-a_{n-1}) \br.
\end{align*}
By the previous Lemma the last summand equals
\begin{multline*}
\bl[a_0,\ldots,a_{n}]_\pi f\br
     \bl (a_{n+1}-a_0)\cldots (a_{n+1}-a_{n-1})\br \\
+
\bl[a_0,\ldots,a_{n+1}]_\pi f\br
     \bl (a_{n+1}-a_0)\cldots (a_{n+1}-a_{n}) \br
\end{multline*}
which are exactly the two missing summands to complete the
claimed formula for $n+1$.
\end{proof}

\subsection{Noncommutative Taylor formulas} 
\label{ss.noncommutativeTaylor}

\begin{prop} \label{DDNote-P2.6.1} Let $\sA$ be a unital Banach algebra and let
$f\in\sO(U)$ be a holomorphic function in the open set $U\subset \C$.  Denote by
$F:\sA_U\to\sA, a\mapsto f(a)$ the map induced by $f$ via the na\"ive
holomorphic functional calculus Theorem \ref{thm:NHFC}. Then $F$ is holomorphic.
The Taylor formula for $F$ about $a\in\sA_U$ reads
\begin{equation} \label{DDNote-G2.6.1}
  \begin{split}
    f(a+b)  = &\sum_{j=0}^N \DDots {a\pup 0;},{a\pup j;}; f \bl b\cldots b\br\\
&\qquad+ \DDots {a\pup 0;}, { {a\pup N;}, (a+b)\pup N+1;};f \bl b\cldots b\br.
 \end{split}
\end{equation}
The remainder $R_N(a,b):=
\DDots {a\pup 0;}, { {a\pup N;}, (a+b)\pup N+1;};f \bl b\cldots b\br$
converges to $0$ as
$N\to\infty$ for $b$ small enough. Hence we have the series expansion
\begin{equation} \label{DDNote-G2.6.2}
  f(a+b)  = \sum_{j=0}^\infty \DDots {a\pup 0;},{a\pup j;}; f \bl b\cldots b \br
\end{equation}
in a neighborhood of $a\in\sA_U$.

Therefore, the $n$--the derivative of $F$ at $a\in\sA_U$ is given by
\begin{equation} \label{DDNote-G2.6.3}
  D^nF(a)[b_1,\ldots,b_n]=
\sum_{\sigma\in S_n} \DDots {a\pup0;},{a\pup n;};f\bl
    b_{\sigma_1}\cldots b_{\sigma_n}\br.
\end{equation}
Here the sum is taken over the set $S_n$ of all $n$--permutations.
\end{prop}

\begin{example}
We note that this indeed contains the ordinary Taylor formula as a special case.
Namely, if $\sA = \C$, then $\sAn = \C$ and hence for $a,b\in \C$
\begin{equation*}
	\DDots {a\pup0;},{a\pup n;};f ( b\cldots b ) = \frac{1}{n!} f^{(n)}(a)\cdot b^n.
\end{equation*}
\end{example}

\begin{proof} \Eqref{DDNote-G2.6.1} is a consequence of the \nonc\ Newton
interpolation formula Proposition \ref{DDNote-P2.3} applied to
$a_0=\ldots=a_N=a, a_{N+1}=a+b$. To estimate the remainder $R_N(a,b)$ we write
using \Eqref{DDNote-G2.14} and a suitable contour $\Gamma$
\[
  R_N(a,b)= \int_\Gamma f(\gz) \cdot \bl (\gz-a)\ii b\br^{N+1}\cdot (\gz-a-b)\ii
  \, \dsl \gz,
\]
thus for $\|b\|$ small enough
\[
  \| R_N(a,b)\|\le C_1 \bl C_2 \|b\|\br^{N+1} \longrightarrow 0,
\]
if $C_2\|b\|< 1$. This proves \Eqref{DDNote-G2.6.2}.

To prove \Eqref{DDNote-G2.6.3} we note that the $n$-th derivative on the left
hand side of \Eqref{DDNote-G2.6.1} is given by
\[
  D^n F(a)[b_1,\ldots,b_n]= \pl_{s_1}\cldots
  \pl_{s_n}\big|_{s_1=\ldots=s_n=0} f(a+\sum_j s_j b_j).
\]
With regard to \Eqref{DDNote-G2.6.1}, using induction, this indeed equals the
right hand side of \Eqref{DDNote-G2.6.3}.
\end{proof}

In view of \Eqref{DDNote-G2.12} and \Eqref{DDNote-G1.12} we may rewrite the
general term in \Eqref{DDNote-G2.6.1} resp.  \Eqref{DDNote-G2.6.2}  as follows
\begin{equation}
  \begin{split} \label{DDNote-G2.6.4}
[{a\pup 0;},\ldots,{a\pup n;}] f &=  [{a\pup 0;},{a\pup 0;}-\nabla_a\pup 1;,\ldots,
             a\pup0;-\nabla_a\pup1;-\ldots-\nabla_a\pup n;]_\pi f\\
            = & \sum_{\ga\in\N^n} \frac{(-1)^{|\ga|}
           f^{(n+|\ga|)}(a\pup 0;)}{\ga?!} \nabla_a^\ga,
  \end{split}
\end{equation}
with $\nabla_a^\ga:=(\nabla_a\pup 1;)^{\ga_1}\cldots (\nabla_a\pup n;)^{\ga_n}$,
thus using \Eqref{eq_adjoint}
\begin{equation}
  \begin{split} \label{DDNote-G2.6.5}
[{a\pup 0;},&\ldots,{a\pup n;}] f \bl b_1\cldots b_n\br\\
             &=  \sum_{\ga\in\N^n} \frac{(-1)^{|\ga|}\cdot
           f^{(n+|\ga|)}(a)\cdot \ad_a^{\ga_1}(b_1)\cldots
                    \ad_a^{\ga_n}(b_n)}{\ga?!} .
  \end{split}
\end{equation}
For the notation $\alpha ?!$ see \eqref{DDNote-G1.2}.  Analogously,
\begin{equation}
  \begin{split} \label{DDNote-G2.6.6}
[{a\pup 0;},&\ldots,{a\pup n;}] f \\
    &=  [{a\pup n;}+\nabla_a\pup1;+\ldots+\nabla_a\pup n;,
     {a\pup n;}+\nabla_a\pup2;+\ldots+\nabla_a\pup n;,
      \ldots, a\pup n;] f\\
      &=  \sum_{\ga\in\N^n} \frac{f^{(n+|\ga|)}(a\pup n;)}{\ga!?} \nabla_a^\ga,
  \end{split}
\end{equation}
thus
\begin{equation}
  \begin{split} \label{DDNote-G2.6.7}
[{a\pup 0;},&\ldots,{a\pup n;}] f \bl b_1\cldots b_n\br\\
             &=  \sum_{\ga\in\N^n} \frac{
                    \ad_a^{\ga_1}(b_1)\cldots
                  \ad_a^{\ga_n}(b_n)\cdot f^{(n+|\ga|)}(a)}{\ga!\cdot (\ga_1+1)\cdot
                  (\ga_1+\ga_2+2)\cldots (|\ga|+n)} .
  \end{split}
\end{equation}
Summing up, for $f$ as in Proposition \ref{DDNote-P2.6.1} we obtain:

\begin{prop}\label{prop-nonc-Taylor-Expansions}
Under the conditions of Proposition \ref{DDNote-P2.6.1}, the \nonc\  Taylor
expansions are given by
\begin{align}
  f(a+b) & =
  \sum_{n=0}^\infty
  \sum_{\ga\in\N^n} \frac{(-1)^{|\ga|} \cdot f^{(n+|\ga|)}(a)\cdot
                  \ad_a^{\ga_1}(b)\cldots
                  \ad_a^{\ga_n}(b) }{%
                  \ga!\cdot (\ga_n+1)\cdot
                (\ga_{n}+\ga_{n-1}+2)\cldots (|\ga|+n)}
                \label{DDNote-G2.6.8}  \\
        & =
  \sum_{n=0}^\infty
  \sum_{\ga\in\N^n} \frac{
                    \ad_a^{\ga_1}(b)\cldots
                \ad_a^{\ga_n}(b)\cdot f^{(n+|\ga|)}(a)}{\ga!\cdot (\ga_1+1)\cdot
                  (\ga_1+\ga_2+2)\cldots (|\ga|+n)}
               .\label{DDNote-G2.6.9}
\end{align}
\end{prop}
In the context of formal power series these formulas were proved
with a slightly different method in \cite[Sec.~1]{Pay2011}.
Note that in contrast to the abstract
of loc. cit. $f^{(n+|\ga|)}$ is on the right in \Eqref{DDNote-G2.6.9}.
This is consistent, however, with \cite[Theorem 1]{Pay2011}.

\subsection{Exponential and Magnus' Theorem} 
\label{ss.expMagnus}

\subsubsection{The Exponential} 
We apply our previous considerations to the exponential function. The power
series expansion of the divided differences (Proposition \ref{DDNote-P1.3} and
Corollary \ref{DDNote-C1.4}) of the exponential function then read
\begin{align*}
\DDots x_0, x_n; \exp
& = \sum_{\ga\in\N^{n+1}} \mlfrac x^\ga/(|\ga|+n)!; ,\\
\DDots {a, a+x_1}, a+x_n; \exp  & =\, e^a \cdot\DDots {0,x_1}, x_n; \exp \\
&= \sum_{\ga\in\N^n} \mlfrac e^a / (|\ga|+n)!; \, x^\ga, \\
\DDots {a, a+x_1, a+x_1+x_2}, a+x_1+\ldots+x_n; \exp  & =\,
\sum_{\ga\in\N^n} \mlfrac e^a / \ga?!; \, x^\ga.
\end{align*}
These expansions converge for all complex numbers $x_0,\ldots, x_n, a\in\C$.

For a unital Banach algebra $\sA$ and $a, b_1,\ldots,b_n\in\sA$, Proposition
\ref{DDNote-P2.6.1} specializes to the well-known Dyson expansion for the
exponential function (cf. \cite{Dys1949}). Namely, involving the Genochi-Hermite
formula \Eqref{Eq-Genocchi-Hermite}, \Eqref{EQHermite} and Proposition
\ref{DDNote-P2.6.1} we find
\begin{equation*}
	\begin{aligned}
		e^{a+b} &= e^a + \sum_{n=1}^{\infty} \DDots{a\pup 0;},{a\pup n;}; \exp (b\cdot \ldots \cdot b)\\
		&=e^a + \sum_{n=1}^\infty \int_{\Delta_n} \exp \Bl \sum_{j=0}^{n} s_j a^{(j)}\Br (b\cdot \ldots \cdot b)ds\\
		&=e^a + \sum_{n=1}^\infty \int_{\Delta_n} e^{s_0 a}\otimes \ldots \otimes e^{s_n a} (b\cdot \ldots \cdot b)ds\\
		&=e^a + \sum_{n=1}^\infty \int_{\Delta_n} e^{s_0 a}\cdot b \cdot e^{s_1 a} \cdot  \ldots \cdot b \cdot e^{s_n a} ds,\\
	\end{aligned}
\end{equation*}
respectively, the finite Taylor formula \eqref{DDNote-G2.6.1} gives  analogously,
\begin{equation*}
	\begin{aligned}
		e^{a+b} &=e^a + \sum_{n=1}^N \int_{\Delta_n} e^{s_0 a}\cdot b \cdot e^{s_1 a} \cdot  \ldots \cdot b \cdot e^{s_n a} ds\\
		&+\int_{\Delta_{N+1}} e^{s_0 a}\cdot b \cdot  \ldots \cdot e^{s_N a}\cdot b \cdot e^{s_{N+1} (a+b)} ds.
	\end{aligned}
\end{equation*}
Furthermore, the \nonc\ Taylor expansions, Proposition
\ref{prop-nonc-Taylor-Expansions}, give for the $\sA$--valued divided difference
(\Eqref{DDNote-G2.6.5} and \eqref{DDNote-G2.6.7})
\begin{equation}\label{eq-4.10}
\begin{split}
[{a\pup 0;},&\ldots,{a\pup n;}] \exp \bl b_1\cldots b_n\br\\
&=  \sum_{\ga\in\N^n} \frac{(-1)^{|\ga|} \cdot e^a\cdot \ad_a^{\ga_1}(b_1)\cldots \ad_a^{\ga_n}(b_n) }{\ga?!} \;\\
&=  \sum_{\ga\in\N^n} \frac{%
	\ad_a^{\ga_1}(b_1)\cldots\ad_a^{\ga_n}(b_n) \cdot e^a}{\ga!?} \;,
\end{split}
\end{equation}
where the right hand side is absolutely convergent.  In the context of formal
power series this formula was proved in \cite[Theorem 2]{Pay2011}.

In sum, the \nonc\  Newton and Taylor expansions (Proposition \ref{DDNote-P2.3},
\Eqref{DDNote-G2.6.8} and \eqref{DDNote-G2.6.9}) now lead to the following
Peano-Baker and Cambell-Baker-Hausdorff type formulas
(cf. \cite[Proposition 2]{Pay2011}).
\begin{align*}
e^{a+b}  = \, &e^a+ \sum_{j=1}^N \int_{\Simplex_j}
e^{s_0 a} \cdot b\cldots b\cdot e^{s_j a} ds \\
& \quad + \int_{\Simplex_{N+1}}
e^{s_0 a} \cdot b\cldots b\cdot e^{s_N a}\cdot b\cdot e^{s_{N+1}
(a+b)} ds \\
=\,& \sum_{n=0}^\infty \sum_{\ga\in\N^n}
\mlfrac  (-1)^{|\ga|}\cdot e^a\cdot \ad_a^\ga(b) / \ga!?;
=\, \sum_{n=0}^\infty \sum_{\ga\in\N^n}
\mlfrac  \ad_a^\ga(b)\cdot  e^a / \ga?!;,
\end{align*}
where we have used the abbreviation
$\ad_a^\ga(b):=\ad_a^{\ga_1}(b)\cldots \ad_a^{\ga_n}(b)$.

\subsubsection{Magnus' Theorem} 

The calculus developed above leads to a simple proof of Magnus' Theorem
\cite{Mag1954}.  Let $A:I\to\sA$ be a smooth function from the interval
$I\subset\R$ into $\sA$; for convenience assume $0\in I$ and let $Y:I\to \sA$ be
the solution of the initial value problem
\begin{equation}
	Y'(t)= A(t)\cdot Y(t),\quad Y(0)=\oneA.
\end{equation}
Let $\Omega(t):= \log Y(t)$ which is defined  at least for small $t$, $\Omega(0)=0$.

From \Eqref{DDNote-G2.6.3} we infer for a holomorphic function $f$ that
\[
\frac{d}{dt} f\bl Y(t)\br =
[Y(t)\pup0;,Y(t)\pup1;]f(Y'(t)),
\]
hence (we omit the $t$ argument for brevity)
\[\begin{split}
\Omega'
& =\, [Y\pup0;,Y\pup1;] \log ( Y' )
=\, \frac {\log Y\pup 0; - \log Y\pup 1; }{%
Y\pup0; - Y\pup 1;} \bl Y\pup1; A\br\\
& =\, \frac{\Omega\pup0;-\Omega\pup1;}{%
e^{\Omega\pup0;}-e^{\Omega\pup1;}}e^{\Omega\pup1;} \bl A \br=\,
\frac{\ad_{\Omega}}{e^{\ad_{\Omega}}-\Id} (A).
\end{split}\]
Thus $\Omega(t)$ is the solution of the nonlinear initial value problem
\begin{equation} \label{DDNote-GMagnus}
	\Omega' = \frac{\ad_{\Omega}}{e^{\ad_{\Omega}}-\Id} (A)
	= \sum_{n=0}^\infty \frac{B_n}{n!} \ad_\gO^n(A), \quad \Omega(0)=0.
\end{equation}
Recall from \Eqref{DDNote-G2.12-1} that
$\ad_\gO(A)=\nabla_\gO(A)=(\gO\pup0;-\gO\pup1;)(A)$.  \Eqref{DDNote-GMagnus} is
due to Magnus \cite{Mag1954}, for a more recent exposition see
\cite[Sec.~2]{Bla2009}.

\subsection{The holomorphic Rearrangement Lemma} 
\label{ss.analyt.rearrangement}

We present a version of Connes' Rearrangement Lemma
\cite[Lemma 6.2]{ConMos2011}, \cite{Les2014} in the framework of the na\"ive
multivariate holomorphic functional calculus.

Before stating the Theorem we introduce some notation.  From \cite{Les2014}
recall: for $a\in\sA$ put $A:= e^a$ and, \cf \Eqref{DDNote-G2.11},
\eqref{DDNote-G2.12},
$\Delta_a^{(j)} := \exp(- \nabla_a^{(j)})$,
$1\le j\le n$. Then
\begin{equation}
   A\pup j; = A\pup 0; \modul  1;\cldots \modul  j;,
      \quad  j\ge 1. \label{EQTREAR5}
\end{equation}

Furthermore, for some $0< \delta< \pi$ denote by $S_\delta$ the strip
$S_\delta := \bigsetdef{ z\in\C}{ |\Im z| < \delta}$ and by $\Lambda_\delta$ the
sector
$ \Lambda_\delta := \bigsetdef{ z\in \C\setminus\{0\} }{ |\arg z| < \delta }.$

Clearly, $\exp( S_\delta ) \subset \Lambda_\delta$. Hence by the Spectral
Mapping Theorem if $a\in \sA$ with $\spec(a)\subset S_\delta$ we have
$ \spec( A=e^a ) \subset \Lambda_\delta$ and
$ \spec( A\ii ) \subset \Lambda_\delta$.

We add in passing that it follows from Gelfand Theory that for commuting
elements $x, y$ in a Banach algebra one has
$\spec( x\cdot y) \subset \spec(x) \cdot \spec(y)$. This implies in particular
that if $\spec(A) \subset \Lambda_\delta$ then
\begin{equation}
     \spec\bl \modul  1;\cldots \modul  j; \br \subset \Lambda_{2\delta}.
\end{equation}
To see this note that
$\modul  1;\cldots \modul  j;
= A\ii \otimes \oneA\otidots \oneA\otimes A\otimes\ldots$ which is a product of
the two commuting elements $A\ii\otimes \oneA\otimes\ldots$ and
$\oneA\otidots \oneA \otimes A\otimes\oneA\otimes\ldots$ both of which have
spectrum contained in $\Lambda_\delta$.

After these preparations we have

\begin{theorem}[Holomorphic Rearrangement Lemma]\label{TAREAR} Let $\sA$ be a
unital Banach algebra and $a\in\sA, A:=e^a$ with $\spec(a) \subset S_\delta$ for
some $0<\delta<\pi/2$.  Furthermore, let $f_0,\ldots,f_p$ be holomorphic functions
defined in the sector $\Lambda_{2\delta}$ satisfying estimates
\begin{equation}\label{eq-rear1}
\begin{split}
    |f_j( s) | &\le C \cdot |s| ^{-\ga_j}, \quad s\in\Lambda_{2\delta},
                                                 |s| \gg 0,\\
    |f_j( s) | &\le C \cdot |s| ^{-\gb_j}, \quad s\in\Lambda_{2\delta},
                                                 |s| \ll 1
\end{split}
\end{equation}
with $\sum\limits_j\ga_j > 1$, $\sum\limits_j \gb_j <1$.

Then the functions
\begin{equation}\label{eq-rear2}
\begin{split}
     F( s_0,\ldots, s_p ) &:=
         \int_0^\infty f_0(u \cdot s_0)\cldots f_p( u \cdot s_p) \dint u,\\
     G( \gl_1,\ldots, \gl_p )& := \int_0^\infty f_0(u) \cdot
           f_1(u\cdot\gl_1)\cldots f_p( u\cdot \gl_p) \dint u
\end{split}
\end{equation}
are holomorphic in $\Lambda_{2\delta}^{p+1}$ resp. $\Lambda_{2\delta}^p$.
Moreover, for $s = (s_0,\ldots, s_p)\in \Lambda_\delta^{p+1}$ one
has
\begin{equation}\label{eq-rear3}
    F(s ) = s_0\ii\cdot G(s_0\ii\cdot s_1, \ldots, s_0\ii\cdot s_p)
\end{equation}
and for $b_1,\ldots,b_p\in \sA$ the following holds
\begin{equation}\label{EQTREAR1}\begin{split}
\int_0^\infty &f_0(u \cdot A)\cdot b_1\cdot f_1(u\cdot A)
      \cldots b_p\cdot f_p(u \cdot A) \dint u \\
   & = F( A\pup 0;,\ldots, A\pup p; )( b_1\cldots b_p ) \\
   & = A^{-1}\cdot G( \modul 1;, \modul 1;\cdot \modul 2;,\ldots,
          \modul 1;\cldots\modul p;)(b_1\cldots b_p),\\
   &= A^{-1}   \int_0^\infty f_0(u) \cdot f_1(u\cdot \modul 1;)
   \cdot  f_2(u\cdot \modul 1;\modul 2;)\cldots
          f_p(u \cdot \modul 1;\cldots\modul p;) \dint u\\
   & \qquad \qquad
        \bl b_1\cldots b_p\br.
\end{split}
\end{equation}
\end{theorem}
\begin{remark} This Theorem should be compared to \cite[Cor. 3.5]{Les2014}.
The present result is not so much interesting because of its formulation
(which is only very mildly more general than \cite[Lemma 6.2]{ConMos2011})
but rather because of its almost trivial proof. The essence of the
Rearrangement Lemma is the trivial substitution $\int_0^\infty \phi(r\cdot u)
\dint u = r\ii\cdot \int_0^\infty \phi(u)\dint u$.
\end{remark}

\begin{proof}
The Theorem is more or less self-evident. The estimates
\Eqref{eq-rear1} guarantee the existence of the integrals
\Eqref{eq-rear2}. For real $s_0>0$  \Eqref{eq-rear3}
follows by changing variables $\tilde u = u\cdot s_0$ in the first
integral in \Eqref{eq-rear2}. For general $s_0$ it then follows
by the uniqueness of analytic continuation. The domains of
definition of the functions $F$ and $G$ are such that
we may apply the na\"ive holomorphic functional calculus
to the commuting elements $A\pup 0;,\ldots, A\pup p;$
resp. $\modul 1;, \modul 1;\cdot\modul 2;, \ldots,
\modul 1;\cldots \modul p;$. Then formula
\Eqref{EQTREAR1} is now clear.
\end{proof}

\appendix
\section{Some Combinatorial formulas} 
\label{s.Appendix}


\subsection{Notation} 
\label{ss.notation}
We fix some notation which will be used frequently. $\N$ denotes the set of
natural numbers including $0$.  For a real number $x>-1$ we write
$x!=\Gamma(x+1)$. If $x,\ga\in\R^n$ then we write, whenever each term is
defined, $\ga!:=\prod_j \ga_j!  =\prod_j\Gamma(\ga_j+1)$, $x^\ga :=\prod_j
x_j^{\ga_j}$, $|\ga|=\sum_j \ga_j$. $|\cdot|$ should not be confused with the
sum norm; the latter will play no role in this paper.

Furthermore, we introduce the abbreviations
\begin{align}
  \ga!?&:= \ga!\cdot (\ga_1+1)\cdot (\ga_1+\ga_2+2)\cldots (|\ga|+n)
             = \ga!\cdot \prod_{j=1}^n \Bl j+\sum_{l=1}^j \ga_l\Br
                \label{DDNote-G1.1}\\
  \ga?!&:= \ga!\cdot (\ga_n+1)\cdot (\ga_n+\ga_{n-1}+2)\cldots
             (|\ga|+n)\nonumber\\
       & = \ga! \cdot \prod_{j=1}^n \Bl j+\sum_{l=0}^{j-1} \ga_{n-l}\Br.
             \label{DDNote-G1.2}
\end{align}

Multiindices will usually be denoted by greek letters. Sums of the form
$\sum_\ga$ will be over $\ga\in\N^n$ with further restrictions indicated below
the $\sum$ sign.  We use the multiindex notation for multinomial coefficients.
That is
\begin{equation} \label{DDNote-G1.3}
  {\ga\choose \gb} = \prod_j {\ga_j\choose \gb_j} =
    \prod_j \frac{\ga_j!}{\gb_j! (\ga_j-\gb_j)!},
\end{equation}
with the usual restrictions on the parameters. In particular we put
${\ga\choose \gb}$ to $0$ if $\ga_j<\gb_j$ or $\gb_j<0$ for at least one index
$j$.

We denote by
\begin{equation} \label{DDNote-G1.4}
  \Simplex_n=\bigsetdef{t\in\R^n}{0\le t_n\le\ldots\le t_1\le 1}
\end{equation}
the standard simplex. Sometimes, the functions
\begin{align}
 s_0  & = (1-t_1),\; s_1= (t_1-t_2),\ldots, s_n=t_n, \label{DDNote-G1.5} \\
 t_n  & = s_n, \; t_{n-1} = s_{n-1}+s_n,\ldots, t_1 = s_1+\ldots+s_n
\end{align}
on $\Simplex_n$ will be more convenient.  Note that $s_j\ge 0$ and
$\sum_{j=0}^n s_j=1$.  Integration over $\Simplex_n$ will always be with respect
to the measure $dt:=dt_1\ldots dt_n= ds_1\ldots ds_n=:ds$, which
differs from the surface measure by a multiplicative constant.

\subsection{Divided Differences} 
\label{ss.DD}
We recall the defining formulas for the divided differences of a smooth resp.
holomorphic function. Cf.~, e.~g., \cite{deB2005}, \cite{MT1951}.  See also
\cite[Appendix A]{Les2014}.

\subsubsection{}
Let $f$ be a holomorphic function on an open set $U\subset \C$ and consider
$x_0, x_1,\ldots, x_n$ a priori distinct points in $U$. Then one defines
recursively the \emph{divided differences}
\begin{equation}\label{EQDD1}\begin{split}
     [x_0]f &:= f(x_0),\\
     \DDots x_0,x_n; f &:= \mlfrac 1/x_0-x_n;\bl \DDots x_0,x_{n-1}; f
         - \DDots x_1,x_n; f\br.
\end{split}\end{equation}
The first few divided differences are therefore
\begin{small}
\begin{align*}
    [x_0,x_1]f & = \mlfrac f(x_0)/(x_0-x_1);+ \mlfrac
                        f(x_1)/(x_1-x_0);,\\
    [x_0,x_1,x_2] f  & = \mlfrac f(x_0)/(x_0-x_1)(x_0-x_2);+
                         \mlfrac f(x_1)/(x_1-x_0)(x_1-x_2);+
                         \mlfrac f(x_2)/(x_2-x_0)(x_2-x_1);,
\end{align*}
\end{small}
and by induction one shows the explicit formula
\begin{equation}\label{EQDD2}
     \DDots x_0,x_n; f  = \sum_{k=0}^n f(x_k) \cdot
            \prod_{j=0, j\not=k}^n (x_k-x_j)\ii,
\end{equation}
resp. the Genocchi-Hermite integral formula \cite[Sec.~1.6]{MT1951},
\cite[Sec.~9]{deB2005}\footnote{%
According to the historical remarks in \cite[Sec.~9]{deB2005} the formula is due
to Genocchi who communicated it to Hermite in a letter.  }
\begin{equation}\label{EQHermite}\begin{split}
   &\DDots x_0,x_n; f = \int_{\sum\limits_{j=0}^n s_j =1, s_j>0}
       f^{(n)}\bl\sum\limits_{j=0}^n s_j x_j\br ds_1\dots ds_n\\
       &= \int_{0\le t_n\le\ldots \le t_1\le 1}
       f^{(n)}\bl (1-t_1)x_0+\ldots+(t_{n-1}-t_n)x_{n-1}+t_nx_n\br
       dt_1\dots dt_n,
\end{split}\end{equation}
when the closed convex hull of $x_0,\ldots, x_n$ is a subset of $U$.  If
$\gamma$ is a closed curve in the domain of $f$ encircling the points
$x_0,\ldots,x_n$ exactly once then by the Residue Theorem and \Eqref{EQDD2} we
have \cite[Sec.~1.7]{MT1951}
\begin{equation}\label{EQDD3}
   \DDots x_0,x_n; f
      = \mlfrac 1/2\pi i;
          \oint_\gamma f(\gz)\cdot \prod_{j=0}^n (\gz-x_j)\ii d\gz.
\end{equation}

\subsection{Basic computations}
\label{ss.bascicomputations}

The following computations are completely elementary and put here for the
convenience of the reader.  We first calculate the integrals of the power
functions $t^\ga, s^\ga$ (multiindex notation!) over the standard simplex.  Then
we calculate the divided differences of the power function $x\mapsto x^N$,
\cf~\refMilTh.

\begin{lemma} \label{DDNote-L1.1}
 Let $\ga=(\ga_0,\ldots,\ga_n)=: (\ga_0,\ga') \in\R^{n+1}$. Then
\begin{equation} \label{DDNote-G1.6}
\ISimplex n s^\ga ds = \frac{\ga!}{(|\ga|+n)!},\quad \text{ for }
\ga_0,\ldots ,\ga_n > -1,
\end{equation}
and
\begin{equation} \label{DDNote-G1.7}
\ISimplex n t_1^{\ga_1}\cldots t_n^{\ga_n} dt =
\frac{(|\ga|+n+1) \ga!}{\ga?!} = \frac{\ga'!}{ \ga' ?!},
\end{equation}
for $\ga_n+1>0,\ga_n+\ga_{n-1}+2>0,\ldots, |\ga|+n+1>0$.
\end{lemma}
\begin{proof} 1. For $n=1$ this is the relation between Euler's beta function
and the Gamma function. Namely,
\[
  \ISimplex{1} (1-t)^{\ga_0} t^{\ga_1} dt = \int_{0}^{1} (1-t)^{\alpha_0} t^{\alpha_1} dt =
  \frac{\Gamma(\ga_0+1)\gG(\ga_1+1)}{\gG(|\ga|+2)}
  =\frac{\ga!}{(|\ga|+1) !}.
\]
Proceeding by induction we consider $n+1$ and write the variables in the form
$s=:(s',s_{n+1})=:(s',t), \ga=(\ga',\ga_{n+1})$.  Then changing integration
variables $s' =(1-s_{n+1})\tilde s$ yields
\[\begin{split}
  \ISimplex{n+1} s^\ga ds
  & = \int_0^1
    (1-t)^{|\ga'|+n} \cdot t^{\ga_{n+1}} \,dt \cdot
   \ISimplex n  (\widetilde s)^{\ga'} \, d\tilde s\\
  & = \frac{\ga'!}{(|\ga'|+n)!} \cdot
       \int_0^1 (1-t)^{|\ga'|+n}\cdot t^{\ga_{n+1}} \, dt \;
        \\
  & =  \frac{(|\ga'|+n)! \cdot \ga_{n+1}! \cdot \ga'!}{%
               (|\ga|+n+1)!\cdot (|\ga'|+n)!}
    =  \frac{\ga!}{(|\ga|+n+1)! }.
\end{split}\]

2.
\[\begin{split}
\ISimplex n &t_1^{\ga_1}\cldots t_n^{\ga_n} dt
     = \ISimplex {n-1} \int_0^{t_{n-1}} t_n^{\ga_n} \, dt_n \,
        t_1^{\ga_1}\cldots t_{n-1}^{\ga_{n-1}} \, dt'\\
    & = \frac{1}{\ga_n+1}\;\ISimplex {n-1}
        t_1^{\ga_1}\cldots t_{n-2}^{\ga_{n-2}}
        t_{n-1}^{\ga_n+\ga_{n-1}+1} \, dt'  \\
    & = \frac{1}{(\ga_n+1)\cdot (\ga_n+\ga_{n-1}+2)\cldots
      (\ga_1+\ldots+\ga_n+n)},
\end{split}\]
and the claim follows.
\end{proof}


\begin{prop}\label{DDNote-P1.2} Let $N\in\Z$ be an integer. Then for
  $z_0,\ldots, z_n\in\C\setminus\{0\}$
\begin{equation} \label{DDNote-G1.8}
  [z_0,\ldots, z_n] \id^N
       =\begin{cases}\displaystyle
       \sum\limits_{|\ga|=N-n} z^\ga,& N\ge n,\\
         \displaystyle
         0, & 0\le N<n,  \\
         \displaystyle
     \frac{(-1)^n}{z_0\cldots z_n}
       \sum\limits_{|\ga|=|N|-1} z^{-\ga},& N<0.
   \end{cases}
\end{equation}
In the sums the letter $\ga$ stands for multiindices in $\N^{n+1}$.

The case $N<0$ can alternatively be written as
$ (-1)^n \sum\limits_{\ga_j\le -1, |\ga|=N-n} z^\ga$.
Furthermore,
\begin{equation} \label{EqDDResolvent} 
   [z_0,\ldots,z_n](\gl - \id)\ii = \prod_{j=0}^n (\gl-z_j)\ii.
\end{equation}
\end{prop}
\begin{proof} We use the contour integral formula \refDDcontour. Let $R$ be
large enough such that $|z_0|,\ldots,|z_n|<R$ and consider the integral
\[\begin{split}
I(R)&:= \oint_{|\gz|=R} \gz^N \cdot \prod_{j=0}^n (\gz - z_j)\ii
    \,\dsl \gz,
    \qquad \dsl\gz := \frac{1}{2\pi i} d\gz \\
  & = \Res_{\gz=0} \gz^N\cdot \prod_{j=0}^n (\gz-z_j)\ii +
     [z_0,\ldots, z_n]\id^N.
\end{split}
\]
If $N\ge 0$ the residue vanishes. Moreover, since $I(R)$ is independent of
$R\to\infty$ we find $I(R)=0$ for $N<n$ since then $I(R)\to 0$ as $R\to\infty$.
For $N\ge n$, expanding the integrand yields
\[
  I(R)  = \MultiSum \ga {n+1} \oint_{|\gz|=R}
              \gz^{N-n-|\ga|-1}\,\dsl\gz \cdot z^\ga
        = \sum_{|\ga|=N-n} z^\ga.
\]
If $N<0$ then since $I(R)=0$
\[\begin{split}
 [z_0,\ldots,z_n]\id^N &= -\Res_{\gz=0} (-1)^{n+1}
    \sum_{\ga\ge 0} \gz^{|\ga|+N}\cdot z^{-\ga} \cdot (z_0\cldots z_n)\ii\\
    &=\frac{(-1)^n}{z_0\cldots z_n}
       \sum\limits_{|\ga|=|N|-1} z^{-\ga}.
\end{split}
\]
From the formulas \refDDrecursion\ or \refDDcontour\ for the divided differences
one immediately infers that
\[
  [z_0,\ldots,z_n]f(\gl-\cdot)
  = (-1)^n [\gl-z_0,\ldots,\gl-z_n]f.
\]
We apply this to the function $x\mapsto x\ii$; together with \Eqref{DDNote-G1.8}
we then obtain \Eqref{EqDDResolvent}.
\end{proof}

\subsection{Power series expansion of divided differences} 
\label{ss.powerseriesDD}

\begin{prop}\label{DDNote-P1.3} Let $f$ be a smooth function in a neighborhood
of $0\in \R$.  Then the Taylor series of $\DD^n
f(x_0,\ldots,x_n):=[x_0,\ldots,x_n]f$ about $0$ is given by
\begin{equation} \label{DDNote-G1.10}
\MultiSum{\ga}{n+1} \frac{f^{(n+|\ga|)}(0)}{(|\ga|+n)!} \, x^\ga.
\end{equation}
The radius of convergence of the series \Eqref{DDNote-G1.10} for $DD^nf$ is at
least as large as the radius of convergence of the Taylor series of $f$. If $f$
is holomorphic then so is $\DD^nf$.
\end{prop}

\begin{proof} We apply Lemma \ref{DDNote-L1.1} and the Genocchi-Hermite integral
formula \refDDGenocchi. Then
\[\begin{split}
  [x_0,\ldots,x_n]f
  &= \ISimplex n f^{(n)} \bl \sum_j s_jx_j \br\, ds
      \sim \sum_{l=0}^\infty \frac{f^{(n+l)}(0)}{l!}
      \ISimplex n \bl \sum_j s_j x_j\br^l\, ds\\
  & = \MultiSum \ga {n+1} \frac{f^{(n+|\ga|)}(0)}{|\alpha|!}
      \frac{|\alpha|!}{\ga!} \ISimplex n s^\ga ds \cdot x^\ga
   =
      \MultiSum{\ga}{n+1} \frac{f^{(n+|\ga|)}(0)}{(|\ga|+n)!} \, x^\ga.
\end{split}
\]
Here the Multinomial Theorem was used.  The remaining claims are straightforward
to check.
\end{proof}

\begin{cor} \label{DDNote-C1.4}
  Let $f$ be a smooth function in a neighborhood of
  $a\in\R$. Then the Taylor series of $F_1(x_1,\ldots,x_n):=
  [a,a+x_1,\ldots, a+x_n]f$ about $0$ is given by
\begin{equation} \label{DDNote-G1.11}
\MultiSum{\ga}{n} \frac{f^{(n+|\ga|)}(a)}{(|\ga|+n)!} \, x^\ga.
\end{equation}
The Taylor series of $F_2(x):=[a,a+x_1,a+x_1+x_2,\ldots,
a+|x|]f$ about $0$ is given by
\begin{equation} \label{DDNote-G1.12}
\MultiSum{\ga}{n} \frac{f^{(n+|\ga|)}(a)}{\ga?!} \, x^\ga.
\end{equation}
The radii of convergence of the series \Eqref{DDNote-G1.11},
\eqref{DDNote-G1.12} are as least as large as the radius of
convergence of the Taylor series of $f$.
If $f$ is holomorphic then so are $F_1$ and $F_2$.
\end{cor}

\begin{proof} The first claim follows from Proposition \ref{DDNote-P1.2} since
\[
[a, a+x_1,\ldots, a+x_n]f = [0,x_1,\ldots,x_n]f(a+\cdot)
  = \MultiSum{\ga}{n} \frac{f^{(n+|\ga|)}(a)}{(|\ga|+n)!} \, x^\ga.
\]

For the proof of \Eqref{DDNote-G1.12} we proceed as in the proof of
Proposition \ref{DDNote-P1.3} but use the second \Eqref{DDNote-G1.7} in Lemma
\ref{DDNote-L1.1}:
\begin{equation}
  \begin{split}
[a,&a+x_1,a+x_1+x_2,\ldots,a+x_1+\ldots+x_n]f\\
   & = \ISimplex n f^{(n)} (a+(s_1+\ldots +s_n) x_1 + (s_2+\ldots+ s_n)
       x_2 +\ldots + s_n x_n) ds\\
       & = \int_{0\le t_n\le \ldots\le t_1\le 1}
      f^{(n)}(a+\sum_j t_j x_j) dt
       \sim \sum_{l=0}^\infty \frac {f^{(n+l)}(a)}{l!}
         \int_{\ldots} \bl \sum_j t_j x_j\br ^l dt\\
         & = \MultiSum \ga n \frac{f^{(n+|\ga|)}(a)}{\ga!(\ga_n+1)\cldots
       (|\ga|+n)}\; x^\ga
        =
       \MultiSum{\ga}{n} \frac{f^{(n+|\ga|)}(a)}{\ga?!} \, x^\ga.
       \qedhere
  \end{split}
\end{equation}
\end{proof}

\subsubsection{A combinatorial application}
\label{ss.combinatoriaapplication}

We give a combinatorial application of \Eqref{DDNote-G1.10},
\eqref{DDNote-G1.11},  and \eqref{DDNote-G1.8}.

Apply \Eqref{DDNote-G1.8} to $[a,a+z_1,\ldots,a+z_n]$. Then
for $N\ge n$
\[\begin{split}
[a,&a+z_1,\ldots,a+z_n]\id^N
   = \sum_{|\ga|\le N-n} a^{N-n-|\ga|} \prod_{j=1}^n
 (a+z_j)^{\ga_j}\\
 & = \sum_{0\le \gb\le \ga, |\ga|\le N-n} a^{N-n-|\gb|} \, z^\gb \,
   {\ga\choose \gb}\\
   & = \sum_{|\gb|\le N-n} a^{N-n-|\gb|} \, z^\gb\, \Bl
      \sum_{\ga\ge\gb, |\ga|\le N-n} {\ga\choose\gb}\Br.
\end{split}\]
The summation is taken over multiindices $\ga,\gb\in\N^n$
with the given restrictions.

On the other hand by Cor. \ref{DDNote-C1.4}
\[\begin{split}
[a,&a+z_1,\ldots,a+z_n]\id^N
   = \MultiSum\gb {n}  \frac{\pl_a^{n+|\gb|} a^N}{(n+|\gb|)!}\,
z^\gb\\
   &= \sum_{|\gb|\le N-n} a^{N-n-|\gb|} \, z^\gb\, {N\choose
n+|\gb|}.
\end{split}\]
Equating coefficients we obtain the following multinomial
identities: Given $\gb\in\N^n$ then
  \begin{align}
    \sum_{0\le\gb\le\ga,|\ga|\le m} {\ga\choose\gb}
       & = {m+n\choose |\gb|+n},\label{DDNote-G1.13}\\
    \sum_{0\le\gb\le\ga,|\ga|= m} {\ga\choose\gb}
       & = {m+n-1\choose |\gb|+n-1}.\label{DDNote-G1.14}
  \end{align}
Note that the second Equ. is an immediate consequence of the first
since\\
${m+n\choose |\gb|+n}-{m-1+n\choose |\gb|+n}={m+n-1\choose
|\gb|+n-1}.$

\Eqref{DDNote-G1.14} should be compared to \cite[Table 169]{GKP94}.
In particular, for $n=2$ it reduces to \cite[Eq.~5.26]{GKP94}.

\bibliography{FCDDbib}   
\bibliographystyle{amsalpha-lmp}
\end{document}